\documentclass[a4paper,reqno]{amsart}
\usepackage{amssymb}
\usepackage{amsmath}
\usepackage{amsthm}
\usepackage{graphicx}
\usepackage{texdraw}
\usepackage{amsfonts}
\usepackage{calligra}
\usepackage{hyperref}
\usepackage{mathtools}
\mathtoolsset{showonlyrefs}
\usepackage{soul}
\usepackage{hyperref}

 \allowdisplaybreaks \numberwithin{equation}{section}
\begingroup
\theoremstyle{plain}
\newtheorem{theorem}{Theorem}[section]
\newtheorem{proposition}[theorem]{Proposition}
\newtheorem{lemma}[theorem]{Lemma}
\newtheorem{corollary}[theorem]{Corollary}

\theoremstyle{definition}

\newtheorem{remark}[theorem]{Remark}

\endgroup

\renewcommand{\d}{\mathrm d}
\newcommand{\dx}{{\mathrm d}x}
\newcommand{\dy}{{\mathrm d}y}
\renewcommand{\dh}{{\mathrm d}\mathcal H^{1}}
\newcommand{\dhn}{{\mathrm d}\mathcal H^{d-1}}
\newcommand{\e}{\epsilon}
\newcommand{\re}{\mathbb R}

\renewcommand{\O}{\Omega }

\newcommand{\mint}{-\hskip -1.06 em \int}

\title{On two functionals involving the maximum of the torsion function}
\author{A. Henrot, I. Lucardesi \and G. Philippin}

\begin{document}

\maketitle

\begin{abstract}
In this paper we investigate upper and lower bounds of two shape functionals involving the maximum of the torsion function. More precisely, we consider $T(\O)/(M(\O)|\O|)$ and $M(\O)\lambda_1(\O)
$, where $\O$ is a bounded open set of $\re^d$ with finite Lebesgue measure $|\O|$, $M(\O)$ denotes the maximum of the torsion function, $T(\O)$ the torsion, and $\lambda_1(\O)$ the first Dirichlet eigenvalue. Particular attention is devoted to the subclass of convex sets.
\end{abstract}

\medskip



\section{Introduction}
The two most classical (and most studied) elliptic PDEs are probably the {\it torsion problem}, also known as {\it St-Venant problem}, and the Dirichlet eigenvalue
problem, see \eqref{torsion} and \eqref{eigen} below. Many estimates and qualitative properties have been obtained for these 
classical problems, see for example works by G. P\'olya, G. Szeg\"o, M. Schiffer, L. Payne, J. Hersch, C. Bandle and many others.
In this paper, following these former works, we are interested in finding bounds (if possible optimal) for
quantities involving the maximum of the torsion function.  We have been particularly inspired by two recent works in
\cite{vdBBB} and \cite{Naples}, where the ratio $T(\Omega)\lambda_1(\Omega)/|\Omega|$ has been investigated in a similar way. Here $T(\Omega)$
denotes the torsion, $\lambda_1(\Omega)$ the first Dirichlet eigenvalue, and $|\Omega|$ the volume of $\O$, see Section \ref{sec1.2} for the precise definitions.

\medskip
Le $M(\Omega)$ be the maximum of the torsion function. In this paper
we investigate upper and lower bounds for the shape functionals
\begin{eqnarray}
& F(\O ) : =&\displaystyle{\frac{T(\O )}{M(\O ) |\O |}\,,}\label{def-F}
\\
& G(\O)  := &\displaystyle{M(\Omega)\lambda_1(\Omega)\,,}\label{def-G}
\end{eqnarray}
defined over the bounded open sets $\Omega$ of $\re^d$ with finite Lebesgue measure.
In Section 2, we prove that the obvious upper bound $F(\Omega) \leq 1$ is actually sharp. Then, we show that for
convex domains we have indeed $F(\Omega)\leq 2/3$ and we give more precise lower and upper bounds for regular plane convex domains
in terms of the curvature of their boundaries. In Section 3, we consider the functional $G$. We prove that the easy lower
bound $G(\Omega) \geq 1$ is actually sharp. For convex domains, we recall the lower bound $G(\Omega) \geq \pi^2/8$ obtained
by L. Payne. Finding the optimal upper bound for $G$ seems much more difficult. Using topological derivatives, we prove that no maximizer exists in a wide class of domains.
When we restrict to the class of convex domains,
we can prove existence of an optimal domain but we cannot identify it. In the plane, we suspect that it is the equilateral
triangle (which is definitely better than the disk). At last, we write the shape derivative of $G$ and prove that the
equilateral triangle does not cancel this shape derivative, in other words it is not a critical point among all regular open sets.

\subsection{Notations}
We adopt standard notations for Lebesgue and Sobolev spaces on a bounded open set of $\re^d$, for example $L^2(\Omega)$
and $H^1(\Omega)$ (space of functions in $L^2$ whose derivative, in the sense of distributions, are still in $L^2$). 
The boundary values of a Sobolev function are always intended in the sense of traces. The $(d-1)$-dimensional Hausdorff measure is denoted by $\mathcal H^{d-1}$. 

Given a bounded open set $\Omega\subset \re^d$, we denote by $|\O|$ its Lebesgue measure, by $\mint_\O$ the average integral over it, and by  $\mathcal D(\Omega)$ the space of $\mathcal C^{\infty}$ functions having compact support contained into $\Omega$. The closure
of $\mathcal D(\Omega)$  in $H^1(\Omega)$ is denoted by $H^1_0(\Omega)$. If the open set $\Omega$ has Lipschitz boundary, we denote by $n$ the outer unit normal vector to $\partial\Omega$, defined a.e.\ on the boundary.

Given a point $x\in \mathbb R^d$ and a positive parameter $r>0$, we denote by $B_r(x)$ the ball of radius $r$ centered in $x$, and with $\overline{B}_r(x)$ its closure.

We define the minimal width of a set as the minimal distance between two parallel supporting hyperplanes.
 
We denote by $f^+$ the positive part of a scalar function $f$, namely $f^+(x):=\max\{ f(x),0\}$.

The partial derivative of a scalar function $f$ defined in $\re^d$ with respect to the $i$-th variable is denoted either by $\partial f/\partial x_i$ or by $f_{,i}$; the same notation is used for higher order partial derivatives.

We adopt the convention of summation over repeated indices.

\subsection{First properties}\label{sec1.2}

Given a bounded open set $\O$ of $\re^d$ with finite Lebesgue measure, we denote by $u_\O $ the \textit{torsion function} of $\Omega$, that is, the solution of
\begin{equation}\label{torsion}
\left\{\begin{array}{lll}
-\Delta u = 1 \quad &  \hbox{in }\O
\\
u\in H^1_0(\O )\,,\end{array}
\right.
\end{equation}
and we set 
\begin{equation}\label{def-TM}
T(\O ):= \| u_\O \|_{L^1(\O )}\,,\quad M(\O ):= \|u_\O \|_{L^\infty(\O )}\,.
\end{equation}

It is easy to check that $u_\O $ is $C^\infty$ inside $\Omega$ and non negative in $\overline \O $, thus
$$
T(\O )= \int_\O  u_\O \, \dx \,,\quad M(\O ) = \max_{\O }  u_\O \,.
$$
We denote by $\lambda_1(\O)$ the first eigenvalue of the Dirichlet Laplacian and by $\varphi_\O$ the corresponding (normalized) eigenfunction, that is, the solution of
\begin{equation}\label{eigen}
\left\{\begin{array}{lll}
-\Delta \varphi = \lambda_1(\O)\varphi \quad   \hbox{in }\O
\\
\varphi \in H^1_0(\O )\,,\end{array}
\right.
\end{equation}
with $\|\varphi_\O\|_{L^2(\O)}=1$.

We recall that the functionals $T$ and $\lambda_1$ admit the following variational formulations:
\begin{equation}\label{varform}
T(\O) = \sup_{v\in H^1_0(\O)\setminus\{0\}} \frac{\Big(\int_\O v\, \dx\Big)^2}{\int_\O |\nabla v|^2\, \dx}\,,\quad \lambda_1(\O) =\inf_{v\in H^1_0(\O)\setminus\{0\}}  \frac{\int_\O |\nabla v|^2\, \dx}{\int_\O v^2\, \dx}\,.
\end{equation}

It follows from the homogeneity relations
$$
T(t\O) = t^{d+2}T(\O)\,,\quad M(t\O) = t^2 M(\O)\,,\quad \lambda_1(t\O) = t^{-2}\lambda_1(\O)\,,\quad t>0\,,
$$
that both $F$ and $G$ are scale invariant.

In the sequel, when no ambiguity may arise, we will denote the torsion function and the first eigenfunction of the Dirichlet Laplacian of a given set $\O$ simply by $u$ and $\varphi$, respectively.

\section{Bounds for the functional $F$}

\subsection{The upper bound}\label{upF}
The upper bound $F(\O)\leq 1$ is obvious. Actually, we are going to prove that this bound is sharp.
This is not so intuitive since the equality $F(\O)=1$ is only true for constant functions and clearly a torsion
function of any domain $\O$ is {\it a priori} far to be constant. The idea is to use the theory of homogenization.
Indeed, by performing suitable spherical holes (with the appropriate radius) in a domain $\O$, we are able to get a sequence of torsion
functions which $\gamma$-converges to something which is no longer a torsion function : the ``strange term coming from nowhere'' in the celebrated paper by D. Cioranescu and F. Murat, \cite{Cio-Mur}. Our theorem is the following.
\begin{theorem}\label{sharp upper bound F}
In any dimension, we can find a sequence of domains $\O_\e$ such that $F(\O_\e) \to 1$.
\end{theorem}

Here we recall the construction of a sequence of perforated domains introduced by Cioranescu-Murat in \cite{Cio-Mur},
see also \cite{Kac-Mur} for a more precise estimate and convergence result. 

Let $\O \subset \re^d$, $d\geq 2$ a regular (or a convex) domain, and $C_0>0$ be fixed. For every $\e>0$, consider the ball $T_\e:= B_{r_\e}(0)$ with a radius $r_\e$ which satisfies
\begin{equation}\label{re}
r_\e=\left\lbrace
\begin{array}{ccc}
  C_0 \e^{d/(d-2)} & \mbox{if} & d\geq 3 \\
  \exp(-C_0/\e^{2}) & \mbox{if} & d=2
\end{array}\right.
\end{equation}
and the perforated domain
\begin{equation}\label{oe}
\O _\e:= \O  \setminus \cup_{z\in \mathbb Z^d}(2\e z + \overline{T}_\e)\,.
\end{equation}
Note that the removed holes form a periodic set in the plane, with period $2\e$. Now let $u_\e$ denote the torsion function of the perforated domain $\O_\e$, extended to zero in the holes.
It is proved in \cite{Cio-Mur} that the sequence $u_\e$ converges weakly in $H^1_0(\O )$ (and strongly in $L^2(\O)$)
to the solution $u^*$ of
$$
\left\{\begin{array}{lll}
-\Delta u^* + a u^* = 1 \  &  \hbox{in }\O
\\
u^*\in H^1_0(\O )\,,\end{array}
\right.
$$
where the constant $a$ satisfies
\begin{equation}\label{constant-a}
    a=\left\lbrace
\begin{array}{ccc}
  \frac{C_0^{d-2}}{2^d} \,d(d-2) \omega_d & \mbox{if} & d\geq 3 \\
  \frac{\pi}{2 C_0} & \mbox{if} & d=2\,,
\end{array}\right. 
\end{equation}
and $\omega_d$ is the volume of the unit ball in $\re^d$.
As a consequence we have
\begin{equation}\label{torsion_volume}
    \int_{\O_\e} u_\e\, \dx \to \int_\O u^*\, \dx\,,\quad |\O_\e|\to |\O|\,,\quad \hbox{as }\e \to 0\,.
\end{equation}

\medskip
Now we want to analyze the asymptotic behavior of the $L^\infty$ norm of the functions $u_\e$. We cannot hope for uniform convergence of $u_\e$ to $u^*$, nevertheless
we can prove the convergence of the $L^\infty$ norms:
\begin{theorem}\label{convergence of the norm}
Let $u_\e$ be the torsion functions of the perforated domains $\Omega_\e$ extended to zero in the holes and let $u^*$ be their weak limit in $H^1_0(\O)$.\\
Then, up to a subsequence, $M(\O_\e)=\|u_\e\|_{L^\infty(\O)} \to \|u^*\|_{L^\infty(\O)}$ as $\e \to 0$.
\end{theorem}
\begin{proof}
We are indebted to G. Buttazzo and B. Velichkov of this proof, after a discussion during a meeting in CIRM-Luminy, 21-25 November 2016.

\medskip
First of all, up to a subsequence, we can assume that $u_\e$ converges pointwise almost everywhere to $u^*$:
$$\mbox{for a.e. } x\in \O\,, \quad  u_\e(x) \to u^*(x)\,.$$
Applying this to a ball centered at a point where $u^*$ is maximum, we infer that
\begin{equation}\label{max1}
    \|u^*\|_{L^\infty(\O)} \leq \liminf_\e \|u_\e\|_{L^\infty(\O)}\,. 
\end{equation}
Now let us assume that the inequality in \eqref{max1} is strict. Then we could find two positive numbers $b_1<b_2$ such that
$$\|u^*\|_{L^\infty(\O)} \leq b_1 < b_2 \leq \liminf_\e \|u_\e\|_{L^\infty(\O)}\,.$$
It is proved in \cite[Proposition 3.2.34]{Vel} that for any function $v$ satisfying $\Delta v +1 \geq 0$ in $\O$, the following inequality holds:
\begin{equation}\label{max2}
    \|v\|_{L^\infty(\O)}  \leq C \left(\int_\O v(x)\, \dx\right)^{d/(d+2)},
\end{equation}
where $C$ is a positive constant which only depends on $\O$. From \eqref{max2} with $v:=(u_\e -b)^+$ and $b=(b_1+b_2)/2$, we obtain
\begin{equation}\label{max3}
    \|(u_\e -b)^+\|_{L^\infty(\O)}  \leq C \left(\int_\O (u_\e -b)^+\,\dx\right)^{d/(d+2)}.
\end{equation}
It comes on the one hand
$$\liminf_\e \|(u_\e -b)^+ \|_{L^\infty(\O)} \geq b_2-b= (b_2-b_1)/2 >0\,,$$
while on the other hand, by $L^2$ convergence,
$$\int_\O (u_\e-b)^+ \,\dx \to \int_\O (u^*-b)^+ \,\dx =0\,,$$
contradicting inequality \eqref{max3}.
\end{proof}

\medskip
Now we are in position to prove Theorem \ref{sharp upper bound F}. Let us introduce $v^*:= au^*$.
In the following Lemma, we list some properties of this function.
\begin{lemma}
Let $a\in \re^+$ and $v^*$ be the solution of
\begin{equation}\label{pde-v*}
- a^{-1} \Delta v^* + v^* = 1
\end{equation}
in $H^1_0(\O )$. Then $0< v^*\leq 1$ in $\Omega$ and, in the limit as $a\to +\infty$, $v^*\rightharpoonup 1$ weakly in $L^2(\O)$.
\end{lemma}
\begin{proof}
 The positivity of $v^*$ in $\O$ is a simple consequence of the maximum principle. 
 For $x_0\in \O $ maximum point for $v^*$ it holds $\Delta v^*(x_0)\leq 0$. In particular, for every $x\in \O$ we have
$$
v^*(x)\leq v^*(x_0)\leq -\Delta v^*(x_0) + v^*(x_0) = 1\,,
$$
which proves the upper bound.
Exploiting the optimality of $v^*$ for the functional
\begin{equation}\label{argmin2}
H^1_0(\O) \ni v \mapsto \frac12 \int_\O a^{-1} |\nabla v|^2\,\dx + \int_\O (v^2  - v)\, \dx\,,
\end{equation}
it is easy to see that $v^*$ and $a^{-1/2}\nabla v^*$ are uniformly (with respect to $a$) bounded in $L^2(\O)$ and $L^2(\O;\re^2)$, respectively. The former bound implies that, in the limit as $a\to +\infty$, up to subsequences, $v^*$ weakly converges in $L^2(\O)$ to some $\overline{v}^*$. The latter bound, combined with \eqref{pde-v*}, implies that the weak limit $\overline{v}^*$ is $1$.
\end{proof}

An immediate consequence of the previous Lemma is 
\begin{equation}\label{max4}
    \liminf_\e \|u_\e\|_{L^\infty(\O)} = \|u^*\|_{L^\infty(\O)} \leq \frac{1}{a}\,.
\end{equation}
Therefore, using \eqref{torsion_volume} and \eqref{max4}, we get, for a subsequence:
$$
\limsup_\e F(\O_\e) \geq 
\frac{\int_\O u^*\, \dx }{(\max u^*)|\O|} = \frac{\int_\O v^*\, \dx }{(\max v^*)|\O|} \geq  \frac{\int_\O v^*\, \dx }{|\O|} \to 1\,,
$$
as $a\to +\infty$. This finishes the proof of Theorem \ref{sharp upper bound F}.

\subsection{The upper bound for convex sets}
The maximizing sequence used in the previous section is very specific, thus we can expect
that in the convex case we can significantly improve the upper bound. Indeed, let us prove the following
\begin{theorem}\label{theoubconvex}
Let $\O$ be any bounded convex domain in $\re^d$, then
\begin{equation}\label{upbFconv}
F(\O)\leq \frac{2}{3}\,.
\end{equation}
Moreover, inequality \eqref{upbFconv} is sharp.
\end{theorem}
\begin{proof}
We use the maximum principle for $P$-functions. Following \cite{Pay1} (for regular convex domains) or \cite{Phi-Saf}
(for general convex domains), it is known that the function $\psi:=|\nabla u|^2 +2u$ takes its maximum at a critical
point of $u$, namely at the point where $u$ is maximum (since by strict concavity of $\sqrt{u}$, see e.g.\ \cite{Kor-Le}, it has only one critical point). Therefore, for every $x\in \O$, we have
\begin{equation}\label{P1}
|\nabla u(x)|^2 +2u(x) \leq 2M(\O)\,;
\end{equation}
in particular, integrating \eqref{P1} over $\O$ yields $3 T(\O) \leq 2 M(\O) |\O|$.

\medskip
In order to prove the sharpness of the inequality, let us consider the sequence of rectangles in the plane $\O_n:=(-n,n)\times (0,1)$.
The same construction holds in any dimension $d$, using the sequence of parallelepipeds $\O_n:=(-n,n)^{d-1}\times (0,1)$.
Let us denote by $u_n$ the torsion function of $\O_n$. By the maximum principle, we have
\begin{equation}\label{P2}
    u_n(x,y)\leq \frac{1}{2}\, y(1-y)\,.
\end{equation}
The function $\frac{1}{2}\, y(1-y)$ can be seen as the torsion function of the unbounded strip $\{0<y<1\}$.
Therefore \eqref{P2} implies that $M(\O_n)\leq \frac{1}{8}\,$.\\
Now, in view of \eqref{torsion}, it is easy to check that the torsion admits the variational formulation
\begin{equation}\label{P3}
    -\frac{1}{2}\, T(\O_n) =\min_{v\in H^1_0(\O_n)} \left\{ \frac{1}{2}\,\int_{\O_n} |\nabla v|^2\,\dx - \int_{\O_n}  v\,\dx\right\}\,.
\end{equation}
Let us introduce the function $\psi_n(x)$ defined as
\begin{equation}\label{psin}
\psi_n(x):=\left\lbrace\begin{array}{cll}
                          1  & \quad \mbox{if }&   x\in [-n+1,n-1]  \\
                          n-x & \quad \mbox{if }&  x\in [n-1,n] \\
                          x-n & \quad \mbox{if }  & x\in [-n,-n+1]\\
                          0 & \quad \mbox{if }  & |x|>n 
                        \end{array}\right.
\end{equation}
and let us choose as a test function in \eqref{P3} the function $v(x,y):=\psi_n(x) \frac{1}{2}\, y(1-y)$, which is an element of $H^1_0(\O_n)$.
We immediately get
$$\int_{\O_n}  v\,\dx = \int_{-n}^n \psi_n(x)\,\dx \int_0^1 \frac{1}{2}\, y(1-y)\,\dy = \frac{n}{6}-\frac{1}{12}\geq \frac{n}{6} - 1\,.$$ 
Since $|\psi_n'(x)|=1$ if $x\in (-n,-n+1)\cup(n-1,n)$ and it is 0 otherwise, and $y(1-y)/2 < 1$ for every $y\in (0,1)$, we have 
$$
\int_{\O_n} |\nabla v|^2\,\dx  \leq 2 +  \int_{-n}^n \psi_n^2(x)\,\dx \int_0^1 (\frac{1}{2}\, -y)^2\,\dy = 2+ \frac{n}{6} - \frac{1}{9}\leq \frac{n}{6} + 2\,.
$$ 
Thus
$$-\frac{1}{2}\, T(\O_n) \leq -\frac{n}{12} + 2\,,$$
which implies that
$$F(\O_n)\geq\frac{2}{3} - \frac{16}{n}\,\to 2/3\quad \mbox{when } n\to +\infty\,.$$
\end{proof}
For strictly convex and regular domains in the plane, one can improve this upper bound by the following:
\begin{theorem}\label{upperboundconvexregular}
Let $\O$ be a strictly convex bounded domain of class $C^2$ in $\re^2$. Let us introduce the quantity $\beta$
which depends only on the geometry of $\O$ (actually its curvature $k$):
\begin{equation}
\beta = 2-\frac{1}{4}\Big(\frac{\min_{\partial\Omega} k}{\max_{\partial\Omega}k}\Big)^3 \leq 2\,.
\end{equation}
Then we have
\begin{equation}\label{upbFstconv}
F(\O)\leq \frac{\beta}{\beta +1} \leq \frac{2}{3}\,.
\end{equation}
\end{theorem}
We postpone the proof of this Theorem to Section \ref{seclowerboundconvex}, where the proof for a similar lower bound will also be given at the same time.

\subsection{The lower bound}
Clearly, by the positivity of $T$, $M$, and Lebesgue measure, the infimum of $F$ is greater than or equal to zero. It is easy to show that the lower bound 0 is optimal: consider the sequence of sets
$$
\O _n:= B_1(0) \bigcup_{i=1}^n B_{r_n}(x_i)\,,\quad n\in \mathbb N\,,
$$
with $x_1,\ldots,x_n\neq 0$ distinct points in a compact set and $r_n>0$ a small parameter (whose precise value will be chosen later). In this case, $u_n:=u_{\O _n}$ is the sum of the torsion functions associated to every single connected component of $\O _n$, namely $u_n= \sum_{i=0}^n u_i$ with
$u_0:= u_{B_1(0)}$ and $u_i := u_{B_{n^{-1/4}}(x_i)}$, $i=1\ldots,n$.
Since
$$
u_0(x)= \frac{1-|x|^2}{4}\quad \hbox{and} \quad u_i= \frac{r_n^2 - |x-x_i|^2}{4}\,,
$$
it is easy to see that
$$
T(\O _n) = \frac{\omega_d}{2d(d+2)} ( 1 + n r_n^{d+2} )\,,\quad M(\O _n) = \frac14\,,\quad |\O _n| = \frac{\omega_d}{d} (1 + n r_n^{d})\,,
$$
where $\omega_d$ is the volume of the unit ball in $\re^d$. By taking $r_n= n^{-1/(2d)}$, we infer that
$$
F(\O _n)=  \frac{2}{d+2}\frac{n^{\tfrac{1}{2}-\tfrac{1}{d}} + 1}{n^{-\tfrac{1}{2}} + 1}\sim n^{-\tfrac{1}{d}}\to 0\,,
$$
implying that $\inf F = 0$.

\subsection{The lower bound for convex sets}\label{seclowerboundconvex}
By strict concavity of $\sqrt{u}$ when $\O$ is convex, see e.g.\ \cite{Kor-Le}, it is easy to get a lower bound for convex sets:
\begin{theorem}
Let $\O$ be any bounded convex set in $\re^d$, then 
\begin{equation}\label{ineqconv}
    F(\O)\geq \frac{1}{(d+1)^2}\,.
\end{equation}
\end{theorem}
\begin{proof}
Since $\sqrt{u}$ is concave, its graph is above the cone of basis $\O$ and vertex $M_0$ the maximum point of $\sqrt{u}$.
Therefore, by comparison of the volumes:
$$\int_\O \sqrt{u(x)}\,\dx \geq \frac{\sqrt{M(\O)} |\O|}{d+1}\,.$$
By taking the square of the previous inequality and using Cauchy-Schwarz inequality for the left-hand side
$$\left(\int_\O \sqrt{u(x)}\,\dx\right)^2 \leq |\O| \int_\O u(x)\,\dx\,,$$
we get the desired inequality.
\end{proof}
We believe that inequality \eqref{ineqconv} is not optimal. For example, in the plane, we conjecture:\\
{\bf Conjecture}: For any plane convex domain, the following lower bound holds:
\begin{equation}\label{conjecturelower}
    F(\O)\geq \frac{1}{3}\,.
\end{equation}
Moreover, this inequality should be optimal, and a minimizing sequence could be a sequence of isosceles triangles
degenerating to a segment.
Let us remark that when $u$ is concave, for instance when $\Omega$ is an ellipse, we obtain exactly in the same way
\begin{equation}
\frac{T(\Omega)}{|\Omega|M(\O)}\geq \frac{1}{3}\,.
\end{equation}
Sufficient conditions on the geometry of $\Omega$ to insure the concavity of $u$  have been established by Kosmodem'yanskii in \cite{Ko}.

\medskip
Let us conclude this section with a theorem in the spirit of Theorem \ref{upperboundconvexregular}.
\begin{theorem}\label{lowerboundconvexregular}
Let $\O$ be a strictly convex bounded domain of class $C^2$ in $\re^2$.
Then we have
\begin{equation}\label{lobFstconv}
F(\O)\geq \frac{1}{4}\left(\frac{{\min_{\partial\Omega} k }}{\max_{\partial\Omega} k}\right)^3,
\end{equation}
where $k$ is the curvature of $\O$.
\end{theorem}
Set for brevity $\alpha:= \big(\min_{\partial\Omega} k / \max_{\partial\Omega} k\big)^3/4$. Note that inequality \eqref{lobFstconv} is better than the general inequality \eqref{ineqconv} when $\alpha>1/8$,
which occurs when $\min_{\partial\Omega} k> \max_{\partial\Omega} k/\sqrt[3]{2}$.
In \cite{Pay-Phi}, Payne and Philippin have derived sharp upper bounds for $|\nabla u|$. The goal of this section is to derive new lower bounds for these quantities by using the same approach as in \cite{Pay-Phi} . For the torsion problem one can associate an auxiliary function involving the curvature $k$ of the level lines \{$u$ = const.\}. Properly chosen, the auxiliary function turns out to satisfy some minimum principles, implying the convexity of the level sets of $u$, under suitable convexity assumptions on $\O$. These results have been derived by Makar-Limanov in \cite{Mak-Lim} for the torsion problem 
with the associate function
\begin{equation}\label{def-M}
P(x):= k|\Delta u|^3 + u[(\Delta u)^2 - u_{,ij}u_{,ij}]\,.
\end{equation}

\begin{proof}[Proof of Theorems \ref{upperboundconvexregular} and \ref{lowerboundconvexregular}]
Making use of normal coordinates with respect to the level lines \{$u$ = const.\} we have
\begin{equation}
|\nabla u|^2 = u_{,i}u_{,i} = u_n^2\,,
\end{equation}
\begin{equation}
\Delta u = u_{nn} + k u_n\,,
\end{equation}
\begin{equation}
u_{,ij}u_{,ij} = u_{nn}^2 + k^2 u_n^2 + 2u_{ns}^2\,,
\end{equation}
where an index $n$ stands for the outward normal derivative  and an index $s$ stands  for the derivative along the level lines \{$u$ = const.\}, and $k$ is the curvature of the level lines defined as
\begin{equation}
k:= -\bigg(\frac{u_{,i}}{|\nabla u|}\bigg)_{,i} = \frac{u_{,ij}u_{,i}u_{,j}- |\nabla u|^2 \Delta u}{|\nabla u|^3}\,.
\end{equation}
The Makar-Limanov function $P$ introduced in \eqref{def-M} may be rewritten in terms of normal coordinates as
\begin{equation}
P = |\nabla u|^3 k - 2[ku_n + k^2u_n^2 + u_{ns}^2]\,.
\end{equation}
Makar-Limanov's result is based on the fact that $P$ is super-harmonic. It then follows that $P$ takes its minimum value $P_{\min}$ on $\partial\Omega$, so that the following quadratic inequality for $k$ holds:
\begin{equation}\label{2.6}
 |\nabla u|^3 k - 2u[ku_n + k^2u_n^2 + u_{ns}^2]\geq P_{\min}\,, \quad x\in \Omega\,.
 \end{equation}
Omitting the term containing $u_{ns}^2$ and solving \eqref{2.6} for $k$, we obtain
 \begin{equation}\label{2.7}
 \frac{\Phi}{4u}\{1-\sqrt{1-z}\} \leq k|\nabla u| \leq \frac{\Phi}{4u}\{1 + \sqrt{1-z}\}\,,
 \end{equation}
with
\begin{equation}\label{2.8}
\Phi:=|\nabla u|^2 + 2u\,,
\end{equation}
\begin{equation}
z:= \frac{8P_{\min}u}{\Phi^2}\,.
\end{equation}
We note that $z\leq 1$ in view of the inequality
\begin{equation}
P \leq \frac{1}{8u}\Phi^2 \quad \hbox{in } \Omega\,,
\end{equation}
derived in \cite{Phi-Po}. Multiplying \eqref{2.7} by $-2|\nabla u|\sqrt{u}$, we obtain
\begin{equation}\label{2.11}
-\frac{\Phi|\nabla u|}{2\sqrt{u}}\{1+\sqrt{1-z}\} \leq \sqrt{u}\frac{\partial \Phi}{\partial n} \leq -\frac{\Phi|\nabla u|}{2\sqrt{u}}\{1-\sqrt{1-z}\}\,,
\end{equation}
in view of
\begin{equation}
\frac{\partial \Phi}{\partial n} = -2k|\nabla u|^2\,.
\end{equation}
For convenience we set
\begin{equation}\label{2.13}
\theta:= \frac{\Phi}{\sqrt{u}} = \frac{|\nabla u|^2}{\sqrt{u}} +2\sqrt{u}\,.
\end{equation}
Replacing \eqref{2.13} in \eqref{2.11}, these inequalities reduce to
\begin{equation}
\frac{\partial u}{\partial n}\sqrt{\Big(\frac{\theta}{2}\Big)^2 - 2P_{\min}}  \leq u \frac{\partial \theta}{\partial n} \leq -\frac{\partial u}{\partial n}\sqrt{\Big(\frac{\theta}{2}\Big)^2 - 2P_{\min}}\,,
\end{equation}
which are equivalent to
\begin{equation}\label{2.15}
-\frac{1}{2}\frac{\d u}{u} \leq -\frac{\d \theta}{\sqrt{\theta^2 - 8P_{\min}}} \leq \frac{1}{2}\frac{\d u}{u}\,.
\end{equation}
These inequalities link the functions $u$ and $\theta$ to their differentials along the orthogonal trajectories of the level lines (also called {\it fall lines} of $u$). Rewriting \eqref{2.15} in the form
\begin{equation}
-\frac{1}{2}\d(\log u) \leq -\d(\log[\theta + \sqrt{\theta^2 - 8P_{\min}}]) \leq \frac {1}{2}\d(\log u)
\end{equation}
and integrating from a point $x\in\Omega$ to the maximum point $x_0$ of $u$ along the fall line joining these points, we obtain
\begin{equation}\label{2.17}
\sqrt{\frac{u(x)}{u(x_0)}}\leq \frac{\theta(x) + \sqrt{\theta^2(x) - 8P_{\min}}}{\theta_0 + \sqrt{\theta_0^2 -8P_{\min}}} \leq \sqrt{\frac{u(x_0)}{u(x)}}\,,
\end{equation}
with
\begin{equation}
\theta_0 := \theta(x_0)=2\sqrt{u(x_0)}\,.
\end{equation}
Multiplying \eqref{2.17} by $(\theta_0 + \sqrt{\theta_0^2 -8P_{\min}})\sqrt{u}$, replacing back $\theta$ by $\Phi$, and recalling that $u(x_0)=M(\O)$, we obtain
\begin{equation}\begin{array}[t]{rcl}\label{2.19}
2u\Big(1+\sqrt{1-\frac{2P_{\min}}{M(\O)}}\Big) - \Phi &\leq& \sqrt{\Phi^2-8P_{\min}u}\\[12pt] &\leq& 2 M(\O)\Big(1+\sqrt{1-\frac{2P_{\min}}{M(\O)}}\Big) - \Phi\,.
\end{array}\end{equation}
Squaring \eqref{2.19} and solving for $\Phi$, we obtain
\begin{equation}\begin{array}[t]{rcl}\label{2.20}
\displaystyle\frac{2P_{\min}}{1+\sqrt{1-\frac{2P_{\min}}{M(\O)}}} +u\Big(1+\sqrt{1-\frac{2P_{\min}}{M(\O)}}\Big) \leq \Phi\\[24pt] \leq\displaystyle\frac{2P_{\min}u}{M(\O)\Big(1+\sqrt{1-\frac{2P_{\min}}{M(\O)}}\Big)} +M(\O)\Big(1+\sqrt{1-\frac{2P_{\min}}{M(\O)}}\Big)\,.
\end{array}\end{equation}
Replacing \eqref{2.8} in \eqref{2.20}, after some reduction we obtain the basic inequalities
\begin{equation}\label{2.21}
\tilde\alpha(M(\O) - u) \leq |\nabla u|^2 \leq \tilde\beta (M(\O)- u)\quad \hbox{in } \Omega\,,
\end{equation}
 with
\begin{equation}\label{2.22}
\tilde\alpha := 1-\sqrt{1-\frac{2P_{\min}}{M(\O)}}\,,
\end{equation}
\begin{equation}\label{2.23}
\tilde\beta :=1+\sqrt{1-\frac{2P_{\min}}{M(\O)}}\,.
\end{equation}
The upper bound for $|\nabla u|^2$ in \eqref{2.21} was already derived in \cite{Pay-Phi}. The lower bound is nontrivial only for strictly convex $\Omega$, whereas the upper bound makes sense even for nonconvex $\Omega$. However in this case $P_{\min}$ is negative, and $\tilde\beta$ is greater than two. We note that inequalities \eqref{2.21} are exact when $\tilde\alpha = \tilde\beta$, i.e.\ when $2P_{\min}(M(\O))^{-1} = 1$. This is the case if and only if $\Omega$ is a disk. For practical use of \eqref{2.21} a computable positive lower bound for the quantity $ 2P_{\min}(M(\O))^{-1}$ is needed. To this end, we write
\begin{equation}\label{2.24}
P_{\min}= \min_{\partial\Omega}(|\nabla u|^3k) \geq  \big(\min_{\partial\Omega} |\nabla u|\big)^3\big( \min_{\partial\Omega}k\big)
\end{equation}
and make use of the inequalities
\begin{equation}\label{2.25}
 \min_{\partial\Omega}|\nabla u| \geq \frac{1}{2  \max_{\partial\Omega} k}\,,
\end{equation}
\begin{equation}\label{2.26}
M(\O) = \max_{\O} u \leq \frac{1}{2}\rho^2 \leq \frac{1}{2}  \big(\min_{\partial\Omega} k\big)^{-2}\,,
\end{equation}
derived in \cite{Pay2}, \cite{Pay-Phi2}, where $\rho$ is the inradius of $\Omega$. Using \eqref{2.24}, \eqref{2.25}, and \eqref{2.26}, for strictly convex $\Omega$ we have
\begin{equation}\label{2.27}
\frac{2P_{\min}}{M(\O)} \geq \frac{2 \big(\min_{\partial\Omega} |\nabla u|\big)^3\big( \min_{\partial\Omega}k\big)}{M(\O)}\geq \frac{1}{2}\left(\frac{\min_{\partial\Omega} k}{\max_{\partial\Omega}k}\right)^3 =2 \alpha\,.
\end{equation}
Replacing \eqref{2.27} in \eqref{2.22} and in \eqref{2.23}, we obtain the bounds
\begin{equation}\label{2.28}
\tilde\alpha \geq \alpha\,, \quad 
\tilde\beta \leq \beta := 2-\alpha\,,
\end{equation}
in particular, the quantities $\alpha$ and $\beta$ may be used in \eqref{2.21} instead of $\tilde\alpha$ and $\tilde\beta$, respectively.
\medskip
Integrating \eqref{2.21} over $\Omega$ and exploiting the estimates \eqref{2.28}, we obtain the following bounds for $\frac{T(\Omega)}{|\Omega|M(\O)}$:
\begin{equation}
\frac{\alpha}{\alpha + 1} \leq \frac{T(\Omega)}{|\Omega|M(\O)} \leq \frac{\beta}{\beta + 1} \leq \frac{2}{3}\,.
\end{equation}
The upper bound proves Theorem \ref{upperboundconvexregular}.

A better lower bound for $\frac{T(\Omega)}{|\Omega|M(\O)}$ may be derived by integrating the inequality
\begin{equation}
P( x) = u_{,ij}u_{.i}u_{.j} - |\nabla u|^2\Delta u + u[(\Delta u)^2 - u_{.ij}u_{,ij}] \geq P_{\min}
\end{equation}
over $\Omega$. Making use of
\begin{equation}
\int_{\Omega}u_{,ij}u_{,i}u_{,j}\, \dx = - \int_{\Omega}u\, u_{,ij}u_{,ij}\, \dx\,,
\end{equation}
we obtain
\begin{equation}
\int_{\Omega}P(x)\, \dx = \int_{\Omega}u[(\Delta u)^2 - 2u_{,ij}u_{,ij}]\, \dx + T(\Omega) \geq P_{\min}|\Omega|\,.
\end{equation}
Since $(\Delta u)^2 - 2u_{,ij}u_{,ij} \leq 0$, it follows that
\begin{equation}
T(\Omega) \geq P_{\min}|\Omega|\,.
\end{equation}
This inequality, together with \eqref{2.27}, gives the lower bound $F(\O) \geq\alpha$, concluding the proof of Theorem \ref{lowerboundconvexregular}.
\end{proof}

\section{Bounds for the functional $G$}

\subsection{The upper bound}
Here we gather the known upper bounds for $G$. In \cite[Theorem 1]{VC}, the authors showed that for every bounded open set $\O\subset \re^d$
$$
G(\O)\leq 3 d \ln 2 + 4\,.
$$
Recently, such estimate was improved by Vogt: in \cite[Theorem 1.5]{V}, the author, exploiting semigroups techniques, proved that 
$$
G(\O) \leq \frac{d}{4} + \frac{1}{4} \sqrt{5(1 + \ln 2 / 4)} \sqrt{d} + 1\,.
$$

Finding the optimal upper bound suggests to look at the shape optimization problem:
\begin{equation}\label{optimG}
\mathcal{P}_G \qquad \sup \{G(\Omega)\,, \ \Omega \subset \re^d\}.
\end{equation}
Even if it looks like as a standard shape optimization problem, the existence of a solution is not clear for us. We believe that a maximizer does not exist and a partial result in this direction is given by the following
\begin{proposition}\label{top-der}
Let $\O\subset \re^d$ be a bounded open set of class $C^2$. Assume that $G$ is differentiable at $\O$ (i.e., the shape derivative of $G$ at $\O$ exists and is given by \eqref{G'2}). Then $\O$ is not a maximizer for $G$.
\end{proposition}
The proof is postponed to \S \ref{ss-sd} and is based on a topological derivative argument: under suitable regularity assumptions on $\O$, removing a small hole near the boundary makes $G$ decrease.

In order to further investigate $\mathcal{P}_G$, other useful tools are represented by numerical tests and the theory of shape derivatives. The former technique suggests that, in the case of polygons in the plane, the optimum should be non convex.
The latter, that we detail in \S \ref{ss-sd}, provides a necessary condition for critical shapes; in particular, it turns out that the equilateral triangle, even if strictly better than the disk, is not optimal for $\sup G$ in dimension $d=2$ (see Corollary \ref{cor-T} below).

Let us now consider the restricted class of convex domains, for which the equilateral triangle could be a maximizer.

\subsection{The maximization problem in the convex framework}
Unlike what happens in the general case, if we add the convex constraint, the existence of a maximizer for $G$ is guaranteed.

\begin{theorem}
The shape functional $G$ admits a maximizer in the class of bounded convex sets of $\mathbb R^d$.
\end{theorem}
\begin{proof}
Let $\O_n$ be a maximizing sequence of convex subsets of $\mathbb R^d$. By the scale invariance of $G$, without loss of generality, we may assume that the elements of the sequence are all contained in a fixed bounded set $K$. 
By the Blaschke selection theorem, there exists a subsequence (not relabeled) converging to some convex set $\O \subset K$ in the Hausdorff metric.

We claim that the minimal width $w_n$ of $\O_n$ does not vanish as $n\to +\infty$, so that the limit set $\O$ has non-empty interior: choosing a suitable reference frame in $\mathbb R^d$, we may assume that $\O_n$ is contained in the strip $\{ x\in \mathbb R^d\ :\   0\leq x_d\leq w_n \}$; by the maximum principle, it is easy to see that the torsion function $u_n$ of $\O_n$ 
satisfies $u_n(x) \leq {x_d(w_n-x_d)}/{2}$ in $\overline{\O}_n$, in particular
\begin{equation}\label{Mn}
M(\O_n) \leq \frac{w_n^2}{8}\,;
\end{equation}
on the other hand, in \cite[formula (4.8)]{Naples} the authors provide the following upper bound for $\lambda_1$ in terms of the minimal width:
\begin{equation}\label{Ln}
\lambda_1(\O_n) \leq \frac{\pi^2}{w_n^2} \left( 1 + c_n ( 3/2 + 3/2^{1/3} + 2^{1/3})\right)\,,
\end{equation}
with $c_n=c_n(\O_n)$ a positive constant vanishing as $w_n\to 0$ (see (19) in \cite{Naples}). By combining \eqref{Mn}, \eqref{Ln}, and the lower bound \eqref{lbGconvex}, we conclude that $G(\O_n)\to \pi^2/8$. This gives a contradiction, since $\pi^2/8$ is clearly not the maximum of $G$.

Since the functional $G$ is continuous in the class of bounded open convex sets (see \cite{Bu-Bu}, \cite{HP}), with respect to the Hausdorff metric, we conclude that the limit set $\O$ is a maximizer for $G$.
\end{proof}

The problem of finding an optimal set is still open. We conjecture that in dimension $d=2$ the maximizer of $G$ among the convex sets is the equilateral triangle $T$, namely, for every $\O$ convex, 
$G(\O)\leq \frac{4}{27} \pi^2=G(T)$ (see the computations in the proof of Corollary \ref{cor-T} below).

\subsection{The lower bound}
The lower bound $G(\O)\geq 1$ is obvious: indeed, making use of \eqref{torsion} and \eqref{eigen}, we get
$$
\int_\O \varphi\, \dx =  - \int_\O \varphi \Delta u \, \dx= -\int_\O u \Delta \varphi \, \dx= \lambda_1(\O)\int_\O u \varphi \, \dx \leq G(\O) \int_\O \varphi\, \dx\,.
$$

Exploiting the same strategy used for the upper bound of $F$, we show that the constant 1 is sharp. 
\begin{theorem}
In any dimension, we can find a sequence of domains $\Omega_\e$ such that $G(\Omega_\e)\to 1$.
\end{theorem}
\begin{proof}
Let $C_0>0$ be fixed. For every $\e>0$, consider the perforated domain $\Omega_\e$ defined in \eqref{oe}, obtained by removing to a given regular set $\Omega$ periodic spherical holes of period $2\e$ and radius $r_\e$ (function of $C_0$, see \eqref{re}).
Let $A_\e:L^2(\O)\to L^2(\O)$ be the resolvent operator of the Dirichlet Laplacian on $\Omega_\e$, which associates to $f\in L^2(\O)$ the unique solution $u\in H^1_0(\Omega_\e)$ to $-\Delta u = f$, extended by zero outside $\Omega_\e$. 

By applying Theorem 2.5 in \cite{Kac-Mur}, we infer that, for every $f\in L^2(\O)$, $A_\e(f)$ strongly converges to $A(f)$ in $L^2(\O)$, where $A$ is the resolvent operator of $ -\Delta + a$ in $H^1(\O)$ with Dirichlet boundary conditions, being $a$ (function of $C_0$) defined in \eqref{constant-a}. 
In particular, in view of \cite[Theorem 2.3.2]{H}, the eigenvalues of $A_\e$ converge to the corresponding eigenvalue of $A$; in other words, we have
$$
\lambda_1(\Omega_\e)\to \lambda_1(\O) + a\,,
$$
as $\e\to 0$. On the other hand, as already noticed in \eqref{max4},
we have $\liminf_\e M(\O_\e)\leq \frac{1}{a}$.
Thus 
\begin{equation}\label{rhs}
\liminf_\e G(\O_\e) = \liminf_\e \lambda_1(\Omega_\e) M(\O_\e) \leq \frac{1}{a}  ( \lambda_1(\O) + a) = 1 + \frac{\lambda_1(\O)}{a}\,.
\end{equation}
By choosing a suitable $C_0$ (vanishing in the case of $d=2$ and diverging to $+\infty$ in the case of $d\geq 3$), the parameter $a$ can be taken arbitrarily large, so that the right-hand side of \eqref{rhs} is arbitrarily close to 1.
This fact, together with the trivial lower bound $G\geq 1$, concludes the proof.
\end{proof}

\subsection{The lower bound for convex sets}

In the convex setting, the optimal lower bound for $G$ was provided by Payne in 1981: for every bounded convex domain $\Omega$ of $\mathbb R^d$, we have
\begin{equation}\label{lbGconvex}
G(\Omega)\geq \frac{\pi^2}{8}\,,
\end{equation}
and the inequality is sharp (see Theorem I and formula (3.12) in \cite{Pay3}). 

The optimality of the constant can be checked, e.g., by considering the sequence of parallelepipeds $\Omega_n:=(-n,n)^{d-1}\times (0,1)$. Indeed, as already seen in the proof of Theorem \ref{theoubconvex}, by comparing the torsion function of $\Omega_n$ with the function $x_{d}(1-x_d)/2$ we get 
\begin{equation}\label{emme}
M(\Omega_n)\leq 1/8\,;
\end{equation}
on the other hand, recalling the definition \eqref{psin} of $\psi_n$ and taking
$$
v(x_1,\ldots,x_d):= \sin(\pi x_d) \Pi^{d-1}_{j=1} \psi_n(x_j) \in H^1_0(\Omega_n)
$$
as test function in the variational formulation \eqref{varform} of $\lambda_1(\Omega_n)$, we get
\begin{equation}\label{lambda}
\lambda_1(\Omega_n) \leq \frac{\int_{\Omega_n} |\nabla v|^2\, \dx}{\int_{\Omega_n} v^2\, \dx} = \frac{d-1}{n-2/3} + \pi^2\,.
\end{equation}
From \eqref{emme} and \eqref{lambda} we obtain the inequality
$$
G(\Omega_n) \leq  \frac{\pi^2}{8} + \frac{d-1}{8(n-2/3)}\,,
$$
whose right-hand side is arbitrarily close to $\pi^2/8$ as $n\to +\infty$.

\subsection{Optimality conditions via shape derivatives.}\label{ss-sd}

In this section we derive optimality conditions by computing the first order shape derivative of $G$. Namely, given $\Omega\subset \mathbb R^d$ bounded, open, regular or convex, connected set, we study the limit (when the latter exists)
$$
G'(\Omega,V):= \lim_{t\to 0}\frac{G(\Omega_t) - G(\Omega)}{t}\,, 
$$
with $\Omega_t:= (I + t V)(\Omega)$, $I$ being the identity map and $V:\re^d \to \re^d$ an arbitrary $C^1$ vector field.

Recalling that $G(\Omega)=M(\Omega)\lambda_1(\Omega)$, if the shape derivative exists, it reads
\begin{equation}\label{G'}
G'(\Omega,V) = M'(\Omega,V) \lambda_1(\Omega) + M(\Omega) \lambda_1'(\Omega,V)\,.
\end{equation}
It is well known (see, e.g., \cite[Th\'eor\`eme 5.7.1]{HP}) that
\begin{equation}\label{L'}
\lambda_1'(\Omega,V) = - \int_{\partial \Omega} \Big(\frac{\partial \varphi}{\partial n}\Big)^2 V\cdot n \, \dhn\,,
\end{equation}
where  $\varphi$ is the (normalized) first eigenfunction and $n$ denotes the unit outer normal to $\partial \Omega$.
Remark that $\frac{\partial \varphi}{\partial n}$ is well defined as soon as $\Omega$ is regular or convex, since $\varphi \in H^2(\Omega)$
in that case.

The computation of $M'$ is more delicate and requires additional assumptions. 

\begin{proposition}\label{prop-M'}
Let $\Omega\subset \mathbb R^d$ be a bounded open convex set.
Then, for every $V\in C^1(\mathbb R^d)$, the shape derivative of $M$ at $\Omega$ in direction $V$ exists and is given by
$$
M'(\Omega,V) = u'(x_0)\,,
$$
where $x_0$ is a maximum point for $u$ and  $u'$ is the solution of 
$$
\left\{\begin{array}{lll}
\Delta u' = 0 \quad  \hbox{in }\Omega
\\
u' +  \nabla u \cdot V \in H^1_0(\Omega)\,.
\end{array}
\right.
$$
\end{proposition}
\begin{proof}
Let $u_t$ denote the torsion function of $\Omega_t$, for $t>0$, and consider the function $\psi:\mathbb R \times \mathbb R^d \to \mathbb R^d$ defined as
$$
(t,x) \mapsto \psi(t,x):= \nabla u_t(x)\,.
$$
By optimality of $x_0$ for $u$, we have $\psi(0,x_0)= 0$. Moreover, the matrix $D_x \psi(0,x_0)$ is invertible: indeed, setting $v:=\sqrt{u}$, we have
$$
D_x \psi(0,x_0)=Hess_u(x_0) = 2 \sqrt{M(\Omega)} Hess_v(x_0)\,,
$$
and the matrix $Hess_v$ is negative definite everywhere in $\Omega$ (see \cite{Kor-Le}).
Thus, by the implicit function theorem, we infer that, in a neighborhood of $x_0$, for $t$ small enough, there exists a unique $x_t$ such that $\nabla u_t(x_t)=0$; furthermore, $t\mapsto x_t$ is differentiable. 

Note that the critical point $x_t$ of $u_t$ must be a maximum point, so that $M(\Omega_t) = u_t(x_t)$.

We claim that, as $t\to 0$,
\begin{align}
& \frac{u_t(x_0) - u(x_0)}{t}  \to u'(x_0) \,,\label{claim-q1}
\\
& \frac{u_t(x_t) - u_t(x_0)}{t}  \to 0 \,.\label{claim-q2}
\end{align}
Once proved the claims we are done, indeed we have
$$
\frac{M(\Omega_t) - M(\Omega)}{t} = \frac{u_t(x_t) - u(x_0)}{t}  = \frac{u_t(x_t) - u_t(x_0)}{t}  +  \frac{u_t(x_0) - u(x_0)}{t}  \to u'(x_0) \,,
$$
which concludes the proof.

Assertion \eqref{claim-q1} follows by applying the mean value property to the harmonic functions $u_t - u$ and $u'$: choose $R>0$ such that $B_R(x_0)\subset \Omega_t$ for every $t<<1$, then we have
\begin{align*}
\Big|\frac{u_t(x_0) - u(x_0)}{t} - u'(x_0)\Big| &= \Big| \mint_{B_R(x_0)} \frac{u_t(x) - u(x)}{t} - u'(x_0) \, \d x\Big|  \leq C \| (u_t - u)/t - u'\|_{L^2(B_R(x_0))}\,.
\end{align*}
The right-hand side vanishes as $t\to 0$, since the map $t\mapsto u_t \in L^2(\mathbb R^d)$ is differentiable at $0$ with derivative $\frac{\d}{\d t} u \lfloor_{t=0}=u'$ (see, for instance, \cite[Chapter 5]{HP}).

Similarly, property \eqref{claim-q2} follows by combining the mean value property of $\nabla u_t$, the differentiability of $t\mapsto x_t$, and the strong convergence of $u_t$ to $u$ in $H^1(\mathbb R^d)$: 
\begin{align*}
\frac{u_t(x_t) - u_t(x_0)}{t}  & = \nabla u_t (\xi_t) \cdot \frac{x_t - x_0}{t} =\Big( \mint_{B_R(\xi_t)} \nabla u_t(x)\, \d x  \Big) \cdot \frac{x_t - x_0}{t} 
\\ & \longrightarrow \nabla u(x_0) \cdot v_0 = 0\,, \quad \hbox{as }t\to 0\,,
\end{align*}
with $\xi_t$ a suitable intermediate point between $x_0$ and $x_t$, $R$ a positive radius such that $B_R(\xi_t)\subset \Omega_t$ for every $t<<1$, and $v_0$ the derivative $\frac{\d}{\d t} x_t \lfloor_{t=0}$.
\end{proof}

\begin{theorem}\label{thm-G'}
Let $\Omega\subset \mathbb R^d$ be a bounded open convex set. Then, for every $V\in C^1(\mathbb R^d)$, the shape derivative of $G$ at $\Omega$ in direction $V$ exists and is given by
\begin{equation}\label{G'2}
G'(\Omega,V) =\int_{\partial \Omega}   \Big[   \lambda_1(\Omega) \frac{\partial u}{\partial n}\,\frac{\partial \phi_{x_0}}{\partial n} - M(\Omega) \Big(\frac{\partial \varphi}{\partial n}\Big)^2 \Big]  V\cdot n\, \dhn\,,
\end{equation}
where $x_0\in \Omega$ is a maximum point of $u$ and $\phi_{x_0}$ is the (Green function) solution of
$$
\left\{\begin{array}{lll}
-\Delta \phi_{x_0} = \delta_{x_0}\quad &\hbox{in }\mathcal D'(\Omega)
\\
\phi_{x_0} = 0 \quad & \hbox{on }\partial \Omega\,.
\end{array}
\right.
$$
\end{theorem}

\begin{proof}
First, we rewrite in terms of $\phi_{x_0}$ the derivative $M'(\Omega,V)$, whose existence is ensured by Proposition \ref{prop-M'}:
\begin{equation}\label{M'}
M'(\Omega,V) = \langle -\Delta \phi_{x_0}, u'\rangle =  - \int_{\partial \Omega} u' \frac{\partial \phi_{x_0}}{\partial n}\, \dhn  = \int_{\partial \Omega}  \frac{\partial u}{\partial n}\,\frac{\partial \phi_{x_0}}{\partial n}\,V\cdot n\, \dhn \,. 
\end{equation}
Formula \eqref{G'2} follows by combining \eqref{G'}, \eqref{L'}, and \eqref{M'}.
\end{proof}

As a consequence of Theorem \ref{thm-G'} we obtain the following optimality condition: if $\Omega\subset \mathbb R^d$ bounded open convex set is a critical shape for $G$, then 
\begin{equation}\label{opt-cond}
 \lambda_1(\Omega) \frac{\partial u}{\partial n}\,\frac{\partial \phi_{x_0}}{\partial n} - M(\Omega) \Big(\frac{\partial \varphi}{\partial n}\Big)^2  = 0 \quad \hbox{a.e. on } \partial \Omega\,.
\end{equation}

\begin{remark}
Notice that the equality \eqref{opt-cond} is satisfied in average over $\partial \Omega$: indeed this corresponds to check that a deformation field $V$ which agrees with $n$ on $\partial \Omega$ does not change the functional, which is true since the functional $G$ is scale invariant.  
\end{remark}

Now we partially answer a question raised in \cite{vdB} where the author asked whether the disk could be the maximizer for $G$.
\begin{corollary}\label{cor-T}
The equilateral triangle gives a better value than the disk; however it is not a critical shape for $G$.
\end{corollary}
\begin{proof}
Let $T\subset \mathbb R^2$ be the equilateral triangle with side of length 1 and vertexes in $(-1/2, -1/(2\sqrt{3}))$, $(1/2, - 1/(2\sqrt{3}))$, and $(0, \sqrt{3})$, so that the center is at the origin. In this case, $u$, $\varphi$, $M$, and $\lambda_1$ can be explicitly computed and read
\begin{align*}
u(x,y) & = \frac{1}{2\sqrt{3}} \Big(y+\frac{1}{2\sqrt{3}}\Big)  \Big(y - \sqrt{3} x - \frac{1}{\sqrt{3}}\Big)   \Big(y+ \sqrt{3}x -\frac{1}{\sqrt{3}}\Big)  
\\
& = \frac{1}{2\sqrt{3}} \Big(y^3 - 3 x^2 y - \frac{\sqrt{3}}{2} y^2  - \frac{\sqrt{3}}{2} x^2 + \frac{1}{6\sqrt{3}}\Big)
\\
\varphi(x,y) &= \Big(\frac{2}{\sqrt{3}}\Big)^{3/2} \Big[    \sin\Big(  \frac{4\pi}{3} \big( 1- \sqrt{3} y\big)\Big) - 2 \cos(2\pi x)    \sin\Big(  \frac{2\pi}{3} \big( 1- \sqrt{3} y\big)\Big)  \Big]
\\
M(T) &= u(0,0) = \frac{1}{36}\,,\quad \lambda_1(T) = \frac{16}{3}\pi^2\,.
\end{align*}
Therefore $G(T)=4\pi^2/27\simeq 1.4622$ while, for the unit disk $\mathbb{D}$, we have $G(\mathbb{D})=j_{0,1}^2/4 \simeq 1.4458$
which proves the first part of the claim.

Now assume by contradiction that $T$ is a critical shape for $G$. Since the normal derivative of $u$ on $\partial T$ is never zero except at the vertices (where both $\nabla u$ and $\nabla \varphi$ vanish), we may recast the optimality condition \eqref{opt-cond} as
\begin{equation}\label{opt-cond2}
\frac{\partial \phi_{x_0}}{\partial n} =  h  \quad \hbox{a.e. on } \partial T\,,
\end{equation}
with $h:= M(\Omega) \big( \lambda_1(\Omega) \big)^{-1}\big(\frac{\partial u}{\partial n}\big)^{-1} \big(\frac{\partial \varphi}{\partial n}\Big)^2 $. 
In particular, if we multiply both sides by an arbitrary harmonic function $w$ and integrate over $\partial T$, we obtain
\begin{equation}\label{contrad}
- w(0) = \int_{\partial T} w h\, \dh\,.
\end{equation}
By taking as test functions $w\equiv 1$ and $w = Re(z^6)$, we get
$$
\left\{
\begin{array}{lll}
\displaystyle{\int_{\partial T} h \, \dh= - 1}
\medskip
\\
\displaystyle{\int_{\partial T} Re(z^6) h\, \dh  = 0\,.}
\end{array}
\right.
$$
Exploiting the symmetry of the domain and of the functions involved, these conditions can be rephrased as follows:
$$
\left\{
\begin{array}{lll}
\displaystyle{\int_{\Sigma} h \, \dh= - \frac{1}{3}}
\medskip
\\
\displaystyle{\int_{\Sigma} Re(z^6) h\, \dh  = 0\,,}
\end{array}
\right.
$$
where $\Sigma$ denotes the basis of the triangle, i.e. the segment $\Sigma = [-1/2,1/2]\times\{-1/(2\sqrt{3})\}$.
This system may be rewritten as 
\begin{equation}\label{T}
\left\{
\begin{array}{lll}
\displaystyle{\sigma:=\int_0^{1/2} \,\frac{(1 + \cos(2\pi x))^2}{x^2 - 1/4}\, \dx = - \frac{27}{ 8}}
\medskip
\\
\displaystyle{\tau:=\int_0^{1/2} P (x) \frac{(1 + \cos(2\pi x))^2}{x^2 - 1/4}\, \dx = 0 \,,}
\end{array}
\right.
\end{equation}
with
\begin{align*}
P(x)  : &=Re(z^6){\lfloor_{\Sigma}} = (x^6 - 15 x^4 y^2 + 15 x^2 y^4 - y^6 ) \lfloor_{y=-1/(2\sqrt{3})}= x^6 - \frac{5}{4} x^4  + \frac{5}{48} x^2 - \frac{1}{1728}\,.\end{align*}
Making use of the factorization
$$
P(x)=  \big(x^2 - \frac14\big) \big(x^4  - x^2    - \frac{7}{48}\big) - \frac{1}{27} \,,
$$
we obtain 
$$
\tau= \int_0^{1/2} \big(x^4  - x^2    - \frac{7}{48}\big) (1 + \cos(2\pi x))^2\, \dx - \frac{\sigma}{27}\,.
$$
It is easy to check that
\begin{align*}
& \int_0^{1/2} (1 + \cos(2\pi x))^2\, \dx = \frac34\,;
\\
& \int_0^{1/2} x^2 (1+\cos (2\pi x))^2\, \dx = \frac{1}{16} - \frac{15}{32 \pi^2}\,;
\\
& \int_0^{1/2} x^4 (1+\cos(2\pi x))^2\, \dx = \frac{3}{320} - \frac{15}{64 \pi^2} + \frac{189}{128 \pi^4} \,.
\end{align*}
Thus we get
\begin{align*}
\tau + \frac{\sigma}{27}  & = \frac{3}{320} - \frac{15}{64 \pi^2} + \frac{189}{128 \pi^4} -    \frac{1}{16}  +  \frac{15}{32 \pi^2} - \frac{7}{48}\cdot\frac34 \\
& =
\frac{3}{320}  -    \frac{1}{16}   - \frac{7}{64} 
+  \frac{15}{32 \pi^2} - \frac{15}{64 \pi^2}
+ \frac{189}{128 \pi^4}
\\
& =
- \frac{13 }{80}  
+  \frac{15}{64 \pi^2}
+ \frac{189}{128 \pi^4}\,.
\end{align*}
Since $\pi$ is not algebraic, the last relation is in contradiction with \eqref{T}, which in turn implies $\tau + \sigma/27= -1/8$.
Therefore we conclude that the equilateral triangle is not a critical shape. 
\end{proof}
\begin{remark}
We point out that the choice of any test function of the form $w=Re z^n$, for $n=1,\ldots,5$, in \eqref{contrad} does not provide any contradiction. Moreover the numerical values of $\sigma$ and $\tau$ defined in the above proof are not so far
of what appears in \eqref{T}. In some sense, the equilateral triangle is not far from being a critical point.
\end{remark}

We conclude the Section with the proof of Proposition \ref{top-der}.
\begin{proof}[Proof of Proposition \ref{top-der}]
Assume by contradiction that $\O$ is a maximizer for $G$. 
Given $x\in \O$ and $\e>0$ a small parameter, we set $\O_\e(x)$ the perforated domain $\O \setminus \overline{B}_\e(x)$ and we denote by $u_{\e,x}$ the associated torsion function.

In the limit as $\e\to 0$, we have the following asymptotic expansions for $\lambda_1(\O_\e(x))$ and $u_{\e,x}$ in terms of $\lambda_1(\O)$ and $u$ (cf. \cite[\S 1.4]{H} and \cite[Chapter 8]{Maz}):
$$
\lambda_1(\O_\e(x))=\left\{ \begin{array}{lll}
\displaystyle{\lambda_1(\O) + \frac{2\pi}{|\log\e|}\varphi^2(x) + o (1/|\log \e|)}\quad & \hbox{if }d=2
\medskip
\\
\displaystyle{\lambda_1(\O) +  \e^{d-2} (d-2) |S^{d-1}| \varphi^2(x) + o (\e^{d-2})}\quad &\hbox{if }d>2
\end{array}
\right.
$$
$$
u_{\e,x}(y)=\left\{ \begin{array}{lll}
\displaystyle{u(y)- \frac{2\pi}{|\log\e|} u(x) \phi_{x}(y) + o (1/|\log \e|)}\quad &\hbox{if }d=2
\medskip
\\
\displaystyle{u(y) -  \e^{d-2} (d-2) |S^{d-1}|u(x) \phi_{x}(y) + o (\e^{d-2})}\quad &\hbox{if }d>2\,,
\end{array}
\right.
$$
where $|S^{d-1}|$ is the measure of the $(d-1)$-sphere and $\phi_{x}(y)$ is the Green function of the Laplace operator vanishing on the boundary of $\O$.
In particular, choosing $x$ different from the maximum point $x_0$ of $u$ and evaluating $u_{\e,x}$ at $x_0$, we obtain
$$
M(\O_\e(x))\geq \left\{ \begin{array}{lll}
\displaystyle{M(\O)- \frac{2\pi}{|\log\e|} u(x) \phi_{x}(x_0) + o (r(\e))}\quad & \hbox{if }d=2
\medskip
\\
\displaystyle{M(\O) -  \e^{d-2} (d-2) |S^{d-1}|u(x) \phi_{x}(x_0) + o (r(\e))}\quad &\hbox{if }d>2\,,
\end{array}
\right.
$$
where $r(\e) =1/|\log \e|$ if $d=2$ and $\e^{d-2}$ otherwise.
In particular, we obtain the lower bound
\begin{equation}\label{absurd}
G(\O_\e(x)) \geq G(\O) + R(x) r(\e) + o (r(\e))\,,
\end{equation}
with
\begin{equation}\label{def-R}
R(x):= M(\O) \varphi^2(x) - \lambda_1(\O) u(x) \phi_{x_0}(x)
\end{equation}
(here we have used the symmetry of the Green function: $\phi_{x}(x_0)   = \phi_{x_0}(x) $).
To get a contradiction, it suffices to find a point $x$ in which $R(x)>0$. 
Taking $x$ close to the boundary, say $x = x_1 - \delta n(x_1)$ for some $x_1\in \partial \O$ and $0<\delta<<1$, and recalling that  $\varphi$ vanishes on $\partial \O$, we may write
$$
\varphi(x) = - \delta \frac{\partial \varphi}{\partial n}(x_1) + \frac{\delta^2}{2} \frac{\partial^2 \varphi}{\partial n^2}(x_1) + o(\delta^2)\,.
$$
Furthermore, by combining  \eqref{eigen} with the relation $\Delta \varphi = \Delta_{\partial \O} \varphi + H_{\partial \O} \frac{\partial \varphi }{\partial n } + \frac{\partial^2 \varphi}{\partial n^2}$ on $\partial \O$, we get
$$
\frac{\partial^2 \varphi}{\partial n^2}(x_1)= - H_{\partial \O}(x_1) \frac{\partial \varphi }{\partial n }(x_1)\,.
$$
Arguing in the same way for $u$ and $\phi_{x_0}$, we obtain the developments
\begin{align*}
u(x) & =  - \delta \frac{\partial u}{\partial n}(x_1) + \frac{\delta^2}{2} \frac{\partial^2 u}{\partial n^2}(x_1) + o(\delta^2)\,,
\\
\phi_{x_0}(x) & =  - \delta \frac{\partial \phi_{x_0}}{\partial n}(x_1) + \frac{\delta^2}{2} \frac{\partial^2 \phi_{x_0}}{\partial n^2}(x_1) + o(\delta^2)\,,
\end{align*}
and the equalities
\begin{align*}
\frac{\partial^2 u}{\partial n^2}(x_1) & = -1 - H_{\partial \O}(x_1) \frac{\partial u}{\partial n }(x_1)\,,
\\
\frac{\partial^2 \phi_{x_0}}{\partial n^2}(x_1) & = - H_{\partial \O}(x_1) \frac{\partial\phi_{x_0}}{\partial n }(x_1)\,.
\end{align*}
In view of these computations, we infer that 
$$
R(x)=   \big(\delta^2 + \delta^3 H_{\partial \O}(x_1)\big) \left[ M(\Omega) \Big(\frac{\partial \varphi}{\partial n}\Big)^2 -  \lambda_1(\Omega) \frac{\partial u}{\partial n}\,\frac{\partial \phi_{x_0}}{\partial n}\right](x_1)  -  \frac{\delta^3 }{2}\lambda_1(\O)\frac{\partial \phi_{x_0}}{\partial n}(x_1)+ o(\delta^3)\,.
$$
By optimality of $\O$, the equality \eqref{opt-cond} holds true at $x_1$, so that 
$$
R(x)=  -  \frac{\delta^3 }{2}\lambda_1(\O)\frac{\partial \phi_{x_0}}{\partial n}(x_1)+ o(\delta^3)\,.
$$
By the Hopf's principle $\frac{\partial \phi_{x_0}}{\partial n}$ is strictly negative on the boundary $\partial \O$, therefore
$R(x)$ is strictly positive. In particular, in view of \eqref{absurd} we conclude that, if $x$ is close enough to the boundary,
$$
G(\O_\e(x))> G(\O)\,,
$$
which is in contradiction with the maximality of $\O$.
\end{proof}

\medskip\noindent
Antoine Henrot, Institut \'Elie Cartan de Lorraine, UMR 7502, Universit\'e de Lorraine CNRS, email: antoine.henrot@univ-lorraine.fr \\
Ilaria Lucardesi, Institut \'Elie Cartan de Lorraine, UMR 7502, Universit\'e de Lorraine CNRS, email: ilaria.lucardesi@univ-lorraine.fr\\
G\'erard Philippin, D\'epartement de Math\'ematiques, Universit\'e Laval  Qu\'ebec, 
\\ email: gerard.philippin@mat.ulaval.ca


\begin{thebibliography}{99}

\bibitem{vdB}{\sc M.~van den Berg}: Estimates for the Torsion Function  and Sobolev Constants, {\it Potential Anal.} {\bf 36},
607--616 (2012)

\bibitem{vdBBB}{\sc M. van den Berg, G. Buttazzo, B. Velichkov}:
Optimization problems involving the first Dirichlet eigenvalue and the torsional rigidity, New trends in shape optimization, 
Internat. Ser. Numer. Math. {\bf 166}, pp. 19--41. Birkhäuser/Springer, Cham (2015) 

\bibitem{VC}{\sc M.~van den Berg, T.~Carroll}: Hardy inequality and $L^p$ estimates for the torsion function, {\it Bull. Lond. Math. Soc.} {\bf 41}, no. 6, 980--986 (2009)

\bibitem{Naples}{\sc M.~van den Berg, V.~Ferone, C.~Nitsch, C.~Trombetti}: On P\'olya's Inequality for Torsional Rigidity
and First Dirichlet Eigenvalue, {\it Integr. Equ. Oper. Theory} {\bf 86}, 579--600  (2016)

\bibitem{Bu-Bu} {\sc D. Bucur, G. Butazzo}: Variational methods in shape optimization problems,
Progress in Nonlinear Differential Equations and their Applications {\bf 65}. Birkh\"auser, Boston (2005)

\bibitem{Cio-Mur} {\sc D.~Cioranescu, F.~Murat}: A strange term coming from nowhere. Topics in the mathematical modelling of composite materials, Progr. Nonlinear Differential Equations Appl., 31, pp. 45--93. Birkhäuser, Boston (1997)

\bibitem{H}{\sc A.~Henrot}: Extremum problems for eigenvalues of elliptic operators. Birkh\"auser, Basel (2006)

\bibitem{HP}{\sc A.~Henrot, M.~Pierre}: Variation et Optimisation de Formes. Une Analyse
G\'eom\'etrique. Math\'ematiques \& Applications 48. Springer, Berlin (2005)


\bibitem{Kac-Mur}{\sc H.~Kacimi, F.~Murat}: {Estimation de l'erreur dans des probl\`emes de {D}irichlet o\`u apparait un terme \'etrange}, {Partial differential equations and the calculus of variations, {V}ol. {II}}, {Progr. Nonlinear Differential Equations Appl.},
    {\bf 2}, pp. 661--696. Birkh\"auser, Boston (1989)

\bibitem{Kor-Le}{\sc N.J.~Korevaar, J.L.~Lewis}: Convex solutions of certain elliptic equations have constant rank Hessians, {\it{ Arch. Rational Mech. Anal.}} {\bf {97}}, 19--32  (1987)

\bibitem{Ko} {\sc A.A.~Kosmodem'yanskii}: Sufficient conditions for the concavity of the solution of the Dirichlet problem for the equation $\Delta u = -1$, {\it Mat. Zametki} {\bf 42},  no. 4, 537--542 (1987)


\bibitem{Mak-Lim}{\sc L.G.~Makar-Limanov}: Solution of the Dirichlet's problem for the equation $\Delta u = -1$ in a convex region, {\it Math. Notes of the Akademy of Sciences of the USSR} {\bf 9},  52--53  (1971)

\bibitem{Maz}{\sc V.~Maz'ya, S.~Nazarov, B.~Plamenevskij}: Asymptotic theory of elliptic boundary value problems in singularly perturbed domains. Vol. 1. Birkh\"auser, Basel (2000)


\bibitem{Pay1} {\sc L. E. Payne}: Bounds for the maximum Stress in the St-Venant problem, {\it Indian Journal of Mechanics  and Mathematics}, special issue in honor of B. Sen, part 1, 51--59  (1968)

\bibitem{Pay3}{\sc L. E.~Payne}: Bounds for solutions of a class of quasilinear elliptic boundary value problems in terms of the torsion function, {\it {Proc. Roy. Soc. Edinburgh Sect. A}} {\bf {88}},  251--265 (1981) 

\bibitem{Pay2} {\sc L.E.~Payne}: Some special maximum principles with applications to isoperimetric inequalities, Maximum principles and eigenvalue problems in partial differential equations, {\it Pitman Research Notes in Math.} {\bf 175}, (P. W. Schaefer, Ed.), Longman, 15-33 (1988)
    
\bibitem{Pay-Phi2} {\sc L.E.~Payne, G.A.~Philippin}: Some remarks on the problems of elastic torsion an of torsional creep, Some Aspects of Mechanics of Continua, part 1, Jadavpur University, 32--40  (1977)
    
\bibitem{Pay-Phi} {\sc L.E.~Payne, G.A.~Philippin}: Isoperimetric inequalities in the torsion and clamped membrane problems for convex plane domains, {\it SIAM J. Math. Anal.} {\bf 14}, 1154--1162 (1983)

\bibitem{Phi-Po} {\sc G.A.~Philippin, G.~Porru}: Isoperimetric inequalities and overdetermined problems for the Saint-Venant equation, {\it New Zealand J. Math.} {\bf 25}, no. 2, 217--227 (1996)

\bibitem{Phi-Saf} {\sc G.A.~Philippin, A.~Safoui}: On extending some maximum principles to convex domains with nonsmooth boundaries, {\it Math. Methods Appl. Sci.} {\bf 33}, no. 15, 1850--1855  (2010)

\bibitem{Vel}{\sc B.~Velichkov}: Existence and regularity results for some shape optimization problems. Tesi. Scuola Normale Superiore di Pisa (Nuova Serie) \textbf{19}. Edizioni della Normale, Pisa (2015)

\bibitem{V}{\sc H.~Vogt}: $L^\infty$-estimates for the torsion function and $L^\infty$-growth of semigroups satisfying Gaussian bounds, \url{https://arxiv.org/pdf/1611.03676.pdf}

\end{thebibliography}
\end{document}